\documentclass[11pt,english]{amsart}

\usepackage{amsmath,leftidx,amsthm}
\usepackage[all]{xy}
\usepackage{xspace}
\usepackage[psamsfonts]{amssymb}
\usepackage[latin1]{inputenc}
\usepackage{graphicx,color,mathrsfs}
\usepackage{hyperref,appendix}
\usepackage{xcolor}

\newtheorem{theorem}{\bf Theorem}[section]
\newtheorem{proposition}[theorem]{\bf Proposition}

\newtheorem{example}[theorem]{\bf Example}

\newtheorem{lemma}[theorem]{\bf Lemma}
\newtheorem{corollary}[theorem]{\bf Corollary}

\newtheorem{remark}[theorem]{\bf Remark}

\def\C{{\mathbb C}}
\def\N{{\mathbb N}}
\def\R{{\mathbb R}}
\def\Z{{\mathbb Z}}

\def\dd{\mathrm{dd}}
\def\d{\mathrm{d}}
\def\P{\mathbb{P}}

\def\supp{\textup{supp}}

\title{Rigidity of the Julia set for Hénon--Sibony maps }
\author{Gabriel Vigny}
\address{LAMFA UMR 7352, Universit\'e de Picardie Jules Verne, 33 rue Saint-Leu, 80039 AMIENS Cedex 1, FRANCE}
\email{gabriel.vigny@u-picardie.fr}
\thanks{The author's research is partially supported by the ANR project DynAtrois (ANR-24-CE40-1163).}

\begin{document}
\begin{abstract} Let $f$ and $g$ be two Hénon--Sibony maps of $\C^k$. We show that if they have the same forward Julia set, then they share a common iterate, thereby extending Lamy's results from dimension 2.
	\end{abstract}

\maketitle

\medskip

\noindent \textbf{Mathematics~Subject~Classification~(2020): 37A25, 37A50, 37F80}

\section{Introduction}

The classification of polynomial automorphisms of $\C^2$ of Friedland and Milnor \cite{FM} establishes complex generalized Hénon maps as the paradigmatic examples of polynomial automorphisms with positive entropy. In higher dimensions, however, the structure of the group of polynomial automorphisms of $\C^k$ is significantly more complex, and no such central family exists. Sibony \cite{Sibony} proposed regular polynomial automorphisms as a natural generalization of Hénon maps. He demonstrated that their rich dynamics can be developed (e.g., \cite{Sibony, superpotentiels}) in a manner analogous to that of Hénon maps, which have been extensively studied over the last thirty years---most notably by Bedford, Lyubich, and Smillie, as well as Forn\ae ss and Sibony (e.g., \cite{BedfordSmillie1, BedfordLyubichSmillie2, FS}).  We shall refer to these as {\it Hénon--Sibony maps}. Such a map $f$ (resp. $f^{-1}$) is a polynomial automorphism of $\C^k$ of degree $d >1$ (resp. $\delta>1$) whose extension to $\P^k$  possesses an indeterminacy set $I_f^+$ (resp. $I_f^-$) such that $I_f^+ \cap I_f^- =\varnothing$. This condition implies the existence of an integer $1 \leq s \leq k-1$ such that $\dim(I_f^+)=k-s-1$, $\dim( I_f^-) =s-1$, and $d^s= \delta^{k-s}$.

For such a Hénon--Sibony map $f$, in analogy with the dynamics of polynomials in $\C$, one can define the \emph{forward Julia set} $K_f^+$ and the \emph{backward Julia set} $K_f^-$ by
\[
K_f^+ := \{ z \in \C^k \; ; \; \limsup_{n\to +\infty} \| f^n(z)\| < \infty \}
\quad \text{and} \quad
K_f^- := \{ z \in \C^k \; ; \; \limsup_{n\to -\infty} \| f^n(z)\| < \infty \},
\]
which are, respectively, the sets of points with bounded forward and backward orbits. 
One can then construct the \emph{Green functions} $G^+_f:\C^k \to \R^+$ and $G^-_f:\C^k \to \R^+$ as
\[
G^+_f(z) := \lim_{n\to \infty} d^{-n} \log^+\|f^n(z)\|
\quad \text{and} \quad
G^-_f(z) := \lim_{n\to \infty} \delta^{-n} \log^+\|f^{-n}(z)\|,
\]
where $\log^+(x) = \max(\log x, 0)$. The Green functions satisfy $K_f^\pm = \{ G_f^\pm = 0 \}$ and make it possible to apply pluripotential theory to the study of the dynamics of $f$. 
Indeed, $G^+_f$ is a plurisubharmonic (psh) function and it satisfies $G^+_f \circ f = d\, G^+_f$. Consequently, $T_f^+ := \dd^c G^+_f$ is a well-defined positive closed $(1,1)$-current on $\C^k$ of mass $1$ (the mass being computed with respect to the Fubini--Study form on $\P^k$). For any $q \leq s$, one can then define the \emph{Green current} $T^+_{f,q} := (T_f^+)^q$. Similarly, for any $p \leq k - s$, one defines $T^-_{f,p} := (T_f^-)^{p}$. 
These currents can then be used to define the \emph{Green measure} $\mu_f$ of $f$ by
\[
\mu_f := T^+_{f,s} \wedge T^-_{f,k-s}.
\]
The measure $\mu_f$ is an invariant and mixing probability measure, and it is the unique measure of maximal entropy $s \log d$ (\cite{BedfordSmillie3, BedfordLyubichSmillie2} in dimension $2$, and \cite{Sibony, superpotentiels, thelin_regular} in higher dimensions).

 In dimension $1$, Beardon~\cite{Beardon-sym,Beardon-sameJ} proved that polynomials sharing the same Julia set must share an iterate (i.e., $f^n = g^m$ for some $n,m \in \N^*$), up to an affine symmetry of the Julia set. In our setting, two Hénon--Sibony maps that share an iterate automatically have the same Julia sets, the same Green functions, and the same Green measure. This naturally raises the question of whether the converse holds, namely to what extent these dynamically defined objects rigidly determine the map.
 
 In dimension $2$, this question was addressed by Lamy in \cite{Lamy_henon}, where he shows that if two Hénon maps share the same forward Julia set, then they must share an iterate. As a consequence, Dujardin and Favre \cite{DF_Manin} proved that the same conclusion already follows from the equality of the Green measure. Since then, this rigidity phenomenon has been revisited and strengthened in several directions, notably by Bera, Pal, and Verma~\cite{BPV1, Bera}. In higher dimension, related questions have also been investigated by Bera and Pal for polynomial shift-like maps~\cite{BeraPal}.

Both Lamy's and Bera, Pal, and Verma's proofs rely on Jung's theorem \cite{Jung}, either directly or through the lens of Bass-Serre theory \cite{Serre} and the Friedland--Milnor classification \cite{FM}. In the two-dimensional case, Bera, Pal, and Verma also rely on the Böttcher coordinate introduced by Hubbard and Oberste-Vorth \cite{HOV}. However, in higher dimensions, a simple classification of automorphisms is unavailable, even for Hénon--Sibony maps; furthermore, the Green function fails to be pluriharmonic on the basin of attraction of $I_f^-$. This article aims to establish a similar result for Hénon--Sibony maps, providing a novel proof for the classical Hénon case. Our main result is as follows.
 
\begin{theorem}\label{tm_main}
	Let $f$ and $g$ be two Hénon--Sibony maps of  $\C^k$. If $K_f^+=K^+_g$, then $f$ and $g$ share a common iterate; that is, there exist positive integers $n$ and $m$ such that $f^n=g^m$. Furthermore, for a given $f$, there are only finitely many such maps $g$ with the same degree as $f$.
\end{theorem}
In Example~\ref{example_notmeasure}, we show that two Hénon--Sibony maps may possess the same Green measure yet fail to share an iterate.

Let us now describe the organization of the article and outline the main ideas of the proof. In Section~\ref{sec2}, we recall that for Hénon--Sibony maps, the condition $K^+_f = K^+_g$ is equivalent to $G_f^+ = G_g^+$. This classical result allows one to reformulate the initial topological question in a measurable framework. We show that if $G_f^+ = G_g^+$, then the algebraic degrees $d_f$ and $d_g$ satisfy $d_f^n = d_g^m$, so that, up to passing to suitable iterates, one may reduce to the case where $f$ and $g$ have the same degree. In the same vein, we prove that for any $q \in \N$, the map $f^{q} \circ g^{-q}$ is an affine automorphism preserving $G_f^+$. This reduces the problem to showing that there are only finitely many affine automorphisms preserving $G_f^+$. This reduction constitutes the starting point of the approaches developed in \cite{Lamy_henon} and \cite{BPV1}.

In Section~\ref{sec3}, we prove Theorem~\ref{tm_main} under the additional assumption that $G^\pm_f = G^\pm_g$. The key idea is that an affine automorphism $\beta$ preserving both $G^+_f$ and $G^-_f$ must preserve the Green measure, hence its barycenter, as well as the indeterminacy sets. This forces $\beta$ to preserve many analytic sets of the form $f^n(L)$ and $f^{-n}(L')$, where $L$ and $L'$ are analytic sets of codimension $s$ and $k-s$, respectively. A deep convergence result of Dinh and Sibony~\cite{DS_saddle} then shows that, for $n$ large enough, $f^n(L) \cap f^{-n}(L')$ spans $\C^k$, which allows us to conclude.

In Section~\ref{sec4}, we introduce the set
\[
N := \left\{ \beta \in \mathrm{Aff}_k \; ; \; \forall n \in \N, \; f^n \circ \beta \circ f^{-n} \in \mathrm{Aff}_k \right\},
\]
where $\mathrm{Aff}_k$ denotes the group of affine automorphisms of $\C^k$. Although $N$ may be larger than the set of affine automorphisms preserving $G^+_f$, it has the crucial advantage of being an algebraic subgroup of $\mathrm{Aff}_k$. Showing that $N$ is finite will yield the desired result. Arguing by contradiction, we assume that $N$ has positive dimension and thus produce a fibration invariant by $f$ on which $f$ will have both topological entropy $<d^s$ and equal to $d^s$. This argument notably relies on another result of Dinh and Sibony~\cite{superpotentiels} concerning the extremality of the Green current $T^+_{f,s}$.

Finally, in Section~\ref{sec5}, we show that when $s=1$, the unique positive closed $(1,1)$-current of mass $1$ supported on $\{G^+_f= a \}$ is $\dd^c \max (G^+_f, a)$.  This argument was communicated to us by Dinh \cite{Dinh_personal}. In contrast, we point out that this uniqueness may fail to hold when $s>1$. Section~\ref{sec6} is then devoted to the study of an example.

\section{Preliminary facts on Hénon--Sibony maps}\label{sec2}
For a Hénon--Sibony map $f$, the Julia set $K^+_f$, the Green function $G_f^+$ and the Green currents $T^+_f$ and $T^+_{f,s}$ are defined in the introduction. 
\begin{lemma}\label{lem_functions_to_currents}
	Let $f$ and $g$ be two Hénon--Sibony maps of $\C^k$. Let $s =\dim(I_f^-)+1$. Then we have the equivalence:
	\begin{enumerate}
		\item $K^+_f= K^+_g$;
		\item $G_f^+= G^+_g$;
		\item $T^+_f= T^+_g$;
		\item  $\dim(I^+_f) \geq  \dim(I^+_g)$ and $T^+_{f,s}= T^+_{g,s}$.
	\end{enumerate}
Furthermore, in all the above cases, we have $I^+_f = I^+_g$.  
\end{lemma}
\begin{proof}
	Obviously $(1)\Leftarrow\ (2)\Leftrightarrow (3) \Rightarrow (4)$. Recall that, in $\P^k$, $\overline{K^+_f}= K^+_f \cup I^+_f$ \cite{Sibony} so $I^+_f$ is characterized by $K^+_f$.	We have that $(1) \Rightarrow (4)$ is an immediate consequence of \cite[Theorem 5.5.4]{superpotentiels} which states that a positive closed current of bounded mass supported by  $K^+_f$ is proportional to $T^+_{f,s}$.  
	
	Let us show $(4) \Rightarrow (2)$. 
	 Consider a generic linear space $P$ of dimension $s$ so that $P \cap (I^+_f\cup I^+_g)=\varnothing$. On $P$, $(dd^c G_f^+)^s= (dd^c G_g^+)^s$ and both functions $G^+_f$ and $G^+_g$ are in the Lelong class of $P$ (e.g. \cite[Theorem 8.4]{DinhSibony_rigidity}). Uniqueness of the solution of the Monge-Ampère equation in the Lelong class implies that $G^+_f$ and $G^+_g$ differ by a constant on $P$. Moving $P$ implies this constant does not depend on $P$ and it is equal to 0 evaluating at points in $K^+_f$ and $K^+_g $. So we have that $(4) \Rightarrow (2)$ and  $I^+_f = I^+_g$ by the above.\end{proof}

 The following lemma shows that a map is linear by using the growth of the Green function. See \cite{Lamy_henon, BPV1, Cantat_Dujardin} for a similar argument.  
 \begin{lemma}\label{lemma_degrees}
 	Let $f$ and $g$ be two Hénon--Sibony maps of $\C^k$, of degree $d_f$ and $d_g$ such that $G^+_f=G^+_g$, then there exist integers $m$ and $n$ such that $d_f^n=d_g^m$ and 	
 	for any $q\in \N$, $f^{qn}\circ g^{-qm}$ is an affine automorphism that preserves $G_f^+$. 
 \end{lemma}
\begin{proof} Let $f$ and $g$ be as above. 	Let $\delta_f$ and $\delta_g$ denote the algebraic degree of $f^{-1}$ and $g^{-1}$. By Lemma~\ref{lem_functions_to_currents},  one has $I_f^+=I_g^+$, thus   $d_f^s=\delta_f^{k-s}$ and   $d_g^s= \delta_g^{k-s}$.
	  Since 
	 $\{d_f^n d_g^m, \ n,m\in \Z \}$ is a multiplicative subgroup of $\R_{>0}$, replacing $f$ and $g$ by iterates and possibly $f$ by $g$, we assume that $d^s_f\leq d^s_g< 2 d^s_f$. By the above, this implies $\delta^{k-s}_f\leq \delta^{k-s}_g< 2 \delta^{k-s}_f$.  We deduce from that $d_f d_g^{-1}\leq 1$ and $\delta_f^{-1} \delta_g<2$. Using the invariance 
	$G^+_f \circ f = d_f G^+_f$, $G^+_f \circ g^{-1} = d_g^{-1} G^+_f$ (\cite[Remarque 2.2.15]{Sibony}) gives
	\begin{equation}\label{pourG+}		
		G^+_f\circ f \circ g^{-1} = d_f d_g^{-1} G^+_f. 
	\end{equation}
Similarly, we have  $G^-_f\circ f \circ g^{-1} =  \delta_f^{-1} G^-_f\circ g^{-1}$. 
Now, using $G^{-}_f(z)\leq \log^+\|z\|	+O(1)$, which is an immediate consequence of \cite[Théorème 1.6.1]{Sibony}, and using the fact $\log^+\|g^{-1}(z)\|\leq \delta_g \log^+\|z\| +O(1)$ since $g^{-1}$ is a rational map of degree $\delta_g$, imply
	\begin{equation}\label{pourG-}		
 G^-_f\circ f \circ g^{-1} (z)\leq  \delta_f^{-1} \delta_g \log^+\|z\| + O(1).
\end{equation}
Now, we have that $G_f :=\max (G_f^+, G^-_f)= \log^+ \|z\|+O(1)$ (e.g. \cite[Theorem 8.4]{DinhSibony_rigidity}). So combining \eqref{pourG+} and \eqref{pourG-} with $d_f d_g^{-1}\leq 1$ and $\delta_f^{-1} \delta_g<2$ gives
\[ 	 \log^+ \|f\circ g^{-1}(z)\| < 2\log^+ \|f\circ g^{-1}(z)\|+O(1). \]
In particular, the polynomial map $f\circ g^{-1}$ has degree $<2$ so it has degree $1$. This implies that $f\circ g^{-1}$ is an affine automorphism of $\C^k$ and by \eqref{pourG+}, it preserves $G_f^+$. We can replace $f$ by $f^q$ and $g$ by $g^q$ and the same argument apply so for any $q\in \N$,  $f^{q}\circ g^{-q}$ is an affine automorphism that preserves $G_f^+$.
\end{proof}

\section{Hénon--Sibony maps sharing both filled Julia sets.}\label{sec3}
In this section, we consider two Hénon--Sibony maps that share both filled Julia sets $K^+_f$ and $K^-_f$ and prove the main theorem under that additional hypothesis. 

The following lemma is where we use our main ingredient which is a convergence result of Dinh and Sibony \cite{DS_saddle}
\begin{lemma}\label{lemma_affine} Let $f$ be a Hénon--Sibony maps of $\C^k$. Then, the set 
	\[G:= \{\alpha \in \mathrm{Aff}(\C^k) , \alpha^*(T^\pm_f)= T^\pm_f \}\] is a finite group and $\begin{cases}
		L_f:&G \to G\\
		&\alpha \mapsto f\circ \alpha \circ f^{-1}
	\end{cases} $ is a well defined group isomorphism. 
\end{lemma}
\begin{proof}
 The fact that \(G\) is a group is immediate. Let us show that 
 \(L_f : G \to G\) is a well-defined group isomorphism. For this, we only need to check that for any 
 \(\alpha \in G\), we have \(f \circ \alpha \circ f^{-1} \in G\). First, 
 \(f \circ \alpha \circ f^{-1}\) is a polynomial automorphism of \(\mathbb{C}^k\) that still leaves 
 \(T_f^+\) and \(T_f^-\) invariant. In particular, by Lemma~\ref{lem_functions_to_currents}, we have the equalities  
 \(G_f^+ = G_f^+ \circ f \circ \alpha \circ f^{-1}\) and 
 \(G_f^- = G_f^- \circ f \circ \alpha \circ f^{-1}\).  
 Since \(\max(G_f^+,G_f^-)=\log^+\|z\|+O(1)\), we deduce that 
 \(\log^+\|f \circ \alpha \circ f^{-1}(z)\| = \log^+\|z\| + O(1)\).  
 This implies that \(f \circ \alpha \circ f^{-1}\) has algebraic degree \(1\), hence belongs to \(G\).
 
 It remains to show that \(G\) is finite. Pick \(\alpha \in G\). Then  
 \(\alpha^*((T_f^+)^s \wedge (T_f^-)^{k-s}) = (T_f^+)^s \wedge (T_f^-)^{k-s}\), so 
 \(\alpha^*(\mu_f) = \mu_f = \alpha_*(\mu_f)\).  
 Consider the barycenter
 \[
 p := \int z \, \mathrm{d}\mu_f(z) \in \mathbb{C}^k,
 \]
where \(z=(z_1,\dots,z_k)\). Since \(\alpha\) is affine,
  \[
 \alpha(p) = \int \alpha(z)\,\mathrm{d}\mu_f(z)
 = \int z\,\mathrm{d}\alpha_*(\mu_f)(z)
 = p.
 \] 
 Up to an affine change of coordinates, we may assume that \(p=(0,\dots,0)\in\mathbb{C}^k\) is the origin, so that \(G\) is a linear group, i.e. \(G \subset \mathrm{GL}(\mathbb{C}^k)\).  
 Consider homogeneous coordinates on \(\mathbb{P}^k\) given by 
 \(z=[z_0:z_1:\dots:z_k]\), so that \(\mathbb{C}^k\) is identified with \(\{z_0=1\}\) and 
 \(H_\infty := \{z_0=0\}\) is the hyperplane at infinity.  
 Let \(\pi : \mathbb{P}^k \dasharrow H_\infty\) be the meromorphic projection defined by 
 \(\pi([z_0:z_1:\dots:z_k]) = [z_1:\dots:z_k]\).  
 Since \(\alpha^*((T_f^+)^s) = (T_f^+)^s\) and the support of \((T_f^+)^s\) intersected with \(H_\infty\) is equal to \(I_f^+\), any \(\alpha \in G\), viewed as an automorphism of \(\mathbb{P}^k\), fixes \(I_f^+\) (and similarly \(I_f^-\)).  
 As \(\alpha\) is linear, it also fixes  $A^+:=\overline{\pi^{-1}(I_f^+)}$ (resp. $A^-:=\overline{\pi^{-1}(I_f^-)}$) which is an algebraic set of dimension $k-s$ (resp. of dimension $s$).  
 
 By the above, we deduce that for any $n \in \N$, $f^n \circ \alpha \circ f^{-n}(A^+)=A^+$, so  $\alpha(f^{-n}(A^+))= f^{-n}(A^+)$. Similarly, we have $\alpha(f^{n}(A^-))= f^{n}(A^-)$. Now, by \cite[Remark 5.15]{DS_saddle}, we know that points in $f^{-n}(A^+)\cap f^n(A^-)$ are equidistributed with respect to $\mu_f$ as $n\to \infty$. Since $\mu_f$ does not charge pluripolar sets, hence analytic sets, we deduce that the finite set $f^{-n_0}(A^+)\cap f^{n_0}(A^-)$ spans $\C^k$ for a suitably large $n_0$. Now, any $\alpha \in G$ fixes  $f^{-n_0}(A^+)\cap f^{n_0}(A^-)$ and there are only finitely many linear maps fixing $f^{-n_0}(A^+)\cap f^{n_0}(A^-)$,
 so that $G$ is finite. 	 
\end{proof}
\begin{remark}\normalfont While \cite{DS_saddle} proves a much stronger result that what we need (i.e\, $\mathrm{span}(f^{-n_0}(A^+)\cap f^{n_0}(A^-))=\C^k$ for some $n_0$), we have found no simpler argument. Note that we can show that $f^{-n_0}(A^+)\cap A^- \backslash \{(0,\dots,0)\}\neq \varnothing$ and $f^{n_0}(A^-)\cap A^+ \backslash \{(0,\dots,0)\}\neq \varnothing$ for some $n_0$ using arguments similar to \cite[Proposition 7.1]{DF_Manin}. This would be enough to conclude in dimension $2$, where $A^\pm$ are one-dimensional vector spaces but is insufficient in higher dimensions ($\mathrm{span}(f^{-n_0}(A^+)\cap A^-)$ may have dimension $<k-s$). 
\end{remark}

\begin{lemma}\label{lemma_main} Let $f$ and $g$ be two Hénon--Sibony maps of $\C^k$. If $T^\pm_f= T^\pm_g$, then $f$ and $g$ share an iterate.   Furthermore, for a given $f$, there are only finitely many such maps $g$ with the same degree as $f$.
\end{lemma}
\begin{proof}
	Let $f$ and $g$ be as in the lemma. By Lemma~\ref{lemma_degrees}, up to iterating $f$ and $g$, we can assume that $f$ and $g$ have common degree $d$ and for all $n$, we have that $\alpha_n:=f^n\circ g^{-n}$ is an affine automorphism which preserves both $T^+_f$ and $T^-_f$. By Lemma~\ref{lemma_affine} whose notations we use, $\alpha_n \in G$ for all $n$ and $G$ is finite. In particular, there exist integers $n \neq m$ such that $\alpha_n =\alpha_m$. This implies $f^{n-m}=g^{n-m}$. 
	
	Finally, observe that for all $g$ which have the same degree than $f$ and such that  $T^\pm_f= T^\pm_g$, then $f\circ g^{-1}$ is in $G$ which is finite by Lemma~\ref{lemma_affine}. 
\end{proof}

The following example shows that two Hénon--Sibony maps can have the same Green measure while not sharing an iterate. It also gives that a Hénon--Sibony map can preserve the forward Julia set of another Hénon--Sibony map without sharing an iterate.  
\begin{example}\label{example_notmeasure}\normalfont  	Let $\pi_i:\C^4 \to \C^2$ for $i=1,2$ denote the projections defined by $\pi_1(x,y,z,w)=(x,y)$ and $\pi_2(x,y,z,w)=(z,w)$. Let $f$ be a Hénon map of $\C^2$ of degree $d$. Then
	\[
	F(x,y,z,w)=\bigl(f(x,y),\,f(z,w)\bigr)
	\quad \text{and} \quad
	G(x,y,z,w)=\bigl(f(x,y),\,f^{-1}(z,w)\bigr)
	\]
	are both Hénon--Sibony maps of $\C^4$ with $s=2$. The Green currents of order $2$ associated with $F$ and $G$ can be expressed in terms of those of $f$ via
	\[
	T^+_{F,2}= \pi_1^*(T^+_f)\wedge \pi_2^*(T^+_f), 
	\quad
	T^-_{F,2}= \pi_1^*(T^-_f)\wedge \pi_2^*(T^-_f),
	\]
	\[
	T^+_{G,2}= \pi_1^*(T^+_f)\wedge \pi_2^*(T^-_f),
	\quad
	T^-_{G,2}= \pi_1^*(T^-_f)\wedge \pi_2^*(T^+_f).
	\]
	From this, one verifies that
	\[
	\mu_F = \pi_1^*(\mu_f)\wedge \pi_2^*(\mu_f) = \mu_G.
	\]
	Moreover, $G(K^+_F)= K^+_F$. Clearly, $F$ and $G$ commute ($F\circ G = G\circ F$), but they do not share any iterate.		
\end{example}

\section{Proof of Theorem\ref{tm_main}}\label{sec4}
By Lemma~\ref{lemma_degrees}, we can thus assume that $f$ and $g$ have the same degree (so that $n=m=1$). To simplify the notations, we let $d$ be the common degree of $f$ and $g$ and $\delta$ the common degree of $f^{-1}$ and $g^{-1}$. Let $\mathrm{Aff}_k$ denote the variety of all the affine automorphisms of $\C^k$.

\begin{lemma}\label{lemma_group}
	Let $N:=\{\beta \in \mathrm{Aff}_k, \ \forall n \in \N, \ f^n \circ \beta \circ f^{-n}  \in \mathrm{Aff}_k \}$ then $N$ is an algebraic subgroup of $\mathrm{Aff}_k$, for any $q\in \N$, $f^q\circ  g^{-q} \in N$ and $\begin{cases}
		L_f:	&N \to N \\
		&\beta \mapsto f \circ  \beta \circ f^{-1}
	\end{cases}$ is a bijection. 
\end{lemma}  
\begin{proof}

By intersection, \(N\) is an algebraic group. To show that 
\(f^{q}\circ g^{-q} \in N\), observe that, by Lemma~\ref{lemma_degrees}, 
\(f^{q}\circ g^{-q}\) is affine and preserves \(G_f^+\):
\[
G_f^+ \circ f^n \circ f^{q}\circ g^{-q} \circ f^{-n} = G_f^+,
\]
and similarly for \(G_f^-\):
\begin{align*}
	G_f^- \circ f^n \circ f^{q}\circ g^{-q} \circ f^{-n} 
	&= \delta^{-n-q} G_f^- \circ g^{-q} \circ f^{-n} \\
	&\leq \delta^{-n-q} \log^+\| g^{-q} \circ f^{-n}(z)\| + O(1) \\
	&\leq \log^+\|z\| + O(1),
\end{align*}
since \(g^{-q} \circ f^{-n}\) is a polynomial map of degree \(\delta^{n+q}\).  
As in the proof of Lemma~\ref{lemma_degrees}, we deduce that
\[
\log^+\| f^n \circ f^{q}\circ g^{-q} \circ f^{-n}(z)\|
\leq \log^+\|z\| + O(1),
\]
so \(f^n \circ f^{q}\circ g^{-q} \circ f^{-n}\) is an affine automorphism. Finally, the endomorphism
\[
\begin{cases}
	N \to N \\
	\beta \mapsto f \circ \beta \circ f^{-1}
\end{cases}
\]
is both injective and a morphism of varieties, hence it is a bijection.
\end{proof}
Assume now that \(N\) is finite. Then there exist \(q\neq q'\) such that 
\(f^{q}\circ g^{-q} = f^{q'}\circ g^{-q'}\), and therefore \(f\) and \(g\) share an iterate.  
This proves Theorem~\ref{tm_main}.  
Our idea is to argue by contradiction and using \(N\), to produce a fibration invariant by $f$ on which $f$ will have both topological entropy $< \log d^s$ (Lemma~\ref{lemma_entropylower}) and equal to $\log d^s$ (Lemma~\ref{lemma_entropyequal}).

Let $\mathrm{Aff}^1_k$ denote the subgroup of $\mathrm{Aff}_k$ of affine automorphism whose linear part is in $\mathrm{Sl}_k(\C)$. We consider the subgroup $N_1$ of $N$ defined by $N_1= N \cap \mathrm{Aff}^1_k$.

\begin{proposition}\label{prop_invariance}
	The set $N_1$ is finite. 
\end{proposition}
Since $N_1$ is an affine variety, we argue by contradiction and assume that $N_1$ has positive dimension in $\mathrm{Aff}_k$. From this assumption, we will deduce Theorem~\ref{tm_main}. Note that we introduce $N_1$ precisely because it is affine, whereas $N$ may not be; later, we will show that $N$ is in turn finite.\\

In what follows, for \(a\geq 0\), we set
\begin{equation}
	\label{eq_Ga_Ta}
	G_a:= \max(G_f^+, a) \ \mathrm{and} \ T_a:= \dd^c G_a
\end{equation}
so that \(G_a\) is psh and \(T_a\) is a positive closed current of mass \(1\) on \(\mathbb{C}^k\).
\begin{lemma}\label{lemma_harmo}
	Take \(\beta\in N\) and \(a\geq 0\). Then \((\beta^{*}T_a)\wedge T^{+}_{f,s}=0\).
\end{lemma}

\begin{proof}
	Take \(n\in\mathbb{N}\). In \(\mathbb{C}^k\),
	\begin{align*}
		(\beta^{*}T_a)\wedge T^{+}_{f,s}
		&= (f^{n})^{*}(f^{n})_{*}\big((\beta^{*}T_a)\wedge T^{+}_{f,s}\big) \\
		&= (f^{n})^{*}\big((f^{n})_{*}(\beta^{*}T_a)\wedge (f^{n})_{*}T^{+}_{f,s}\big) \\
		&= (f^{n})^{*}\big((f^{n})_{*}(\beta^{*}d^{-n}(f^{n})^{*}T_{d^{n}a})\wedge d^{-sn}T^{+}_{f,s}\big) \\
		&= d^{-(s+1)n}(f^{n})^{*}\big((f^{n}\circ\beta\circ f^{-n})^{*}T_{d^{n}a}\wedge T^{+}_{f,s}\big),
	\end{align*}
	where we used \(T_a = d^{-n}(f^{n})^{*}T_{d^{n}a}\) and \((f^{n})_{*}T^{+}_{f,s} = d^{-sn}T^{+}_{f,s}\). Since \(f^{n}\circ\beta\circ f^{-n}\in\mathrm{Aff}_k\), the current \((f^{n}\circ\beta\circ f^{-n})^{*}T_{d^{n}a}\) has mass \(1\). By Bézout, its wedge product with \(T^{+}_{f,s}\) has mass at most \(1\), and the pullback of a positive closed \((s+1,s+1)\)-current of mass \(1\) by \(f^{n}\) has mass at most \(\delta^{(k-s-1)n}\ll d^{-(s+1)n}\). Hence \((\beta^{*}T_a)\wedge T^{+}_{f,s}=0\).
\end{proof}

Since \((\beta^{*}T_a)\wedge T^{+}_{f,s} = \dd^{c}(G_a\circ\beta\, T^{+}_{f,s})\), the current \(G_a\circ\beta\, T^{+}_{f,s}\) is positive harmonic and supported on \(\{G_f^{+}=0\}\). If \(s=1\) (e.g. when \(k=2\)), one could conclude directly from \cite{DinhSibony_rigidity} that such a current is closed, as \(T_f^{+}\) is \emph{very rigid} in their sense. Here we must verify this and then use the rigidity of the Green current \(T^{+}_{f,s}\) from \cite{superpotentiels}.

\begin{lemma}\label{lemma_constant}
	Take $\beta\in N$, then $G^+_f \circ \beta$ is constant on  the support of  $T^+_{f,s}$.
\end{lemma}
\begin{proof} Assume it is not the case. In particular, we can take $a>0$ such that $H_a:= \min (G^+_f \circ \beta, a) =  G_0 \circ \beta +a-G_a \circ \beta  $ is non constant on  the support of  $T^+_{f,s}$. It is a non negative function and, by difference, Lemma~\ref{lemma_harmo} gives $\dd^c H_a \wedge T^+_{f,s}=0$. To show that    $ H_a \circ \beta  T^+_{f,s}$ is closed, it suffices to show that for any smooth $1$-form $\theta$, compactly supported in $\C^k$ and any Kähler form $\tilde{\omega}$, one has
	\begin{equation} \label{eq_closed}
		\langle   \d H_a \circ \beta \wedge T^+_{f,s}, \theta \tilde{\omega}^{k-s-1}\rangle =0, 
	\end{equation} 
	where 	$\langle   , \rangle$ denotes the usual duality pairing for currents. 	
	By Cauchy-Schwarz inequality:
	\begin{align*} \
		\left| \langle   \d H_a \circ \beta \wedge T^+_{f,s}, \theta \tilde{\omega}^{k-s-1}\rangle \right|^2\leq \\
		\left| \langle    \d H_a \wedge \d^c H_a  \wedge T^+_{f,s},  \tilde{\omega}^{k-s-1}\rangle \right|\left| \langle  i \theta \wedge \bar{\theta} \wedge T^+_{f,s}, \tilde{\omega}^{k-s-1}\rangle \right|  . 
	\end{align*} 	
	The second term of the right-hand side is bounded and we show that the first term vanishes. For that, observe that 
	$\d H_a \wedge \d^cH_a = 1/2 \dd^c H_a^2 - H_a\dd^c H_a $ and since 
	$\dd^c H_a \wedge T^+_{f,s}=0$, we deduce that $\d H_a \wedge \d^cH_a \wedge T^+_{f,s} \wedge   \tilde{\omega}^{k-s-1}= \dd^c H_a^2  \wedge T^+_{f,s} \wedge   \tilde{\omega}^{k-s-1}$ is a non negative measure and
	\[  \left| \langle    \d H_a \wedge \d^c H_a  \wedge T^+_{f,s},  \tilde{\omega}^{k-s-1}\rangle \right| = \left| \langle    \dd^c H_a^2  \wedge T^+_{f,s},  \tilde{\omega}^{k-s-1}\rangle \right|.\]
	It remains to show that the mass of $\dd^c H_a^2  \wedge T^+_{f,s} \wedge   \tilde{\omega}^{k-s-1}$ is equal to  $0$ on $\C^k$. For that take $R\gg 1$ and consider \[\phi_R= \frac{\log \max (\|z\|, R^2) - \log \max (\|z\|, R)}{\log R}\]
	so $ \phi_R$ is equal to $1$ inside $B(0,R)$ and equal to $0$ outside $B(0,R^2)$. By Stokes formula:
	\[ \left| \langle    \dd^c H_a^2  \wedge T^+_{f,s}\wedge   \tilde{\omega}^{k-s-1}, \phi_R \rangle\right|=  \left| \langle   H_a^2  T^+_{f,s}\wedge   \tilde{\omega}^{k-s-1}, \dd^c \phi_R \rangle \right|.\]
	Now $\dd^c \phi_R$ can be written as $T_{1,R} - T_{2,R}$ where $T_{i,R}$ are positive closed currents of mass $1/\log R$. By Bézout, since $H_a^2 \leq a^2$, we deduce   
	\[ \left| \langle    \dd^c H_a^2  \wedge T^+_{f,s}\wedge   \tilde{\omega}^{k-s-1}, \phi_R \rangle\right|\leq   \frac{2 a^2}{\log R} ,\]
	so $\dd^c H_a^2  \wedge T^+_{f,s} \wedge   \tilde{\omega}^{k-s-1}$ is zero on $\C^k$ by letting $R \to \infty$ and \eqref{eq_closed} is proved. 
	
	But now, \cite[Theorem 5.5.4]{superpotentiels} states that a positive closed current of bounded mass supported by  $\{G^+_f=0\}$ is proportional to $T^+_{f,s}$. This implies that  $H_a $ is constant on  the support of  $T^+_{f,s}$, a contradiction. 
\end{proof}

Pick $z \in \supp(\mu_f)$ so $z$ belongs to both $\supp(T^+_{f,s})$ and $\supp(T^-_{f,k-s})$.  
Let $\Phi^+(\beta)$ (resp. $\Phi^-(\beta)$) be the constant such that
\[
G^+_f \circ \beta \, T^+_{f,s}= \Phi^+(\beta)\, T^+_{f,s}
\quad \text{(resp. } 
G^-_f \circ \beta \, T^-_{f,k-s}= \Phi^-(\beta)\, T^-_{f,k-s}\text{)}.
\]
Then
\[
\Phi^+(\beta)= G^+_f(\beta(z)), 
\qquad 
\Phi^-(\beta)= G^-_f(\beta(z)),
\]
independently of the choice of $z \in \supp(\mu_f)$. 
%


\begin{lemma}\label{lemma_otherapproch} Let $\beta \in N$ be such that $	\Phi^+(\beta)=0$, then $\beta^* T^+_f= T^+_f$. Consequently, there exists an integer $M\geq 1$ such that for any $z\in \supp(\mu_f)$, we have $\#\{\beta \in N, \  \beta(z)  \in \supp(\mu_f)\} \leq M$. 
\end{lemma}
\begin{proof} Pick $\beta \in N$ with $	\Phi^+(\beta)=0$, then by Lemma~\ref{lemma_constant}, $G^+_f\circ \beta =0$ on the support of $T^+_{f,s}$. Since $G^+_f\circ \beta= G^+_{\beta^{-1} \circ f\circ \beta}$, this says that $T^+_{f,s}$ is a positive closed  $(s,s)$-current of mass $1$ supported on $\{G^+_{\beta^{-1} \circ f\circ \beta}=0\}$. As above, \cite[Theorem 5.5.4]{superpotentiels} implies that $T^+_{f,s}= T^+_{\beta^{-1} \circ f\circ \beta,s}= \beta^{*}T^+_{f,s}$. Lemma~\ref{lem_functions_to_currents} then gives $\beta^* T^+_f= T^+_f$ as claimed.
	
 By Lemma~\ref{lemma_affine}, the set of $\beta\in N$ such that $\beta^*T^+_f=T^+_f$ and  $\beta^*T^-_f=T^-_f$ is finite, let $M$ denote its cardinality.  By contradiction, assume that there exists $z\in \supp(\mu_f)$ such that $\#\{ \beta \in N, \ \beta(z)  \in \supp(\mu_f)\}>M$. For any such $\beta$, we have $\Phi^+(\beta)= G^+_f(\beta(z))=0$ and $\Phi^-(\beta)= G^-_f(\beta(z))=0$. Hence, $\beta^*T^+_f=T^+_f$ and, symmetrically,  $\beta^*T^-_f=T^-_f$, a contradiction.
\end{proof}

The group $\mathrm{Aff}_k$ embeds into the vector space $\mathrm{End}_{\mathrm{Aff}}(\C^k) \simeq \C^{k^2+k}$ of affine endomorphisms of $\C^k$. After a change of coordinates, we may assume that $0 \in \supp(\mu_f)$. Choose points $z^0,\dots,z^k \in \supp(\mu_f)$ such that $z^0=0$ and $(z^1,\dots,z^k)$ is a basis of $\C^k$ (this is possible since $\supp(\mu_f)$ is not pluripolar). Then any $\beta \in \mathrm{End}_{\mathrm{Aff}}(\C^k)$ is uniquely determined by $\beta(z^0),\dots,\beta(z^k)$, and
\[
\|\beta\|_{\mathrm{Aff}} := \max_{0\le i \le k} \|\beta(z^i)\|
\]
defines a norm on $\mathrm{End}_{\mathrm{Aff}}(\C^k)$.

\begin{lemma}\label{lem_proper}
	The function $\Phi : N_1 \to \R^+$ defined by $\Phi := \max(\Phi^+,\Phi^-)$ is proper and belongs to the Lelong class. More precisely,
	\[
	\Phi(\beta) = \log^+ \|\beta\|_{\mathrm{Aff}} + O(1).
	\]
\end{lemma}

\begin{proof}
	For each $i$, we have
	\[
	\log^+ \|\beta(z^i)\| = \max\bigl(G^+_f(\beta(z^i)),\, G^-_f(\beta(z^i))\bigr) + O(1)	= \Phi(\beta) + O(1).
	\]
	Thus, for $\beta \in N_1$,
	\[
	\log^+ \|\beta\|_{\mathrm{Aff}}
	= \max_{0\le i \le k} \max\bigl(G^+_f(\beta(z^i)),\, G^-_f(\beta(z^i))\bigr) + O(1)
	= \Phi(\beta) + O(1).
	\]
	\end{proof}
For $z \in \C^k$, we define 
\[\mathrm{Orb}(z):= \{ \beta(z), \ \beta\in N_1\}\]
the class of $z$ under the action of $N_1$. From the definition of $N_1$, we deduce $f(\mathrm{Orb}(z))= \mathrm{Orb}(f(z))$.

\begin{corollary}
For $z\in \supp(\mu_f)$, the set $\mathrm{Orb}(z)$ is an affine irreducible subvariety of $\C^k$ of positive dimension. Furthermore, there exist a constant $M>0$ such that for every  $z\in \supp(\mu_f)$, we have $\dim(\mathrm{Orb}(z))$ is independent of $z$ with $\dim(\mathrm{Orb}(z))<k$ and both the degree of $\mathrm{Orb}(z)$  and $\#\mathrm{Orb}(z) \cap  \supp(\mu_f)$ are uniformly bounded from above by $M$.
\end{corollary}
\begin{proof}Pick $z \in \supp(\mu_f)$. The first assertion follows directly from the fact that the map $\beta \mapsto \beta(z)$ is proper in $N_1$, by Lemma~\ref{lem_proper}. By Lemma~\ref{lemma_otherapproch}, we obtain that
	\[
	0 < \dim(\mathrm{Orb}(z)) < k,
	\]
	and that this dimension is independent of $z \in \supp(\mu_f)$ (indeed, it equals $\dim(N_1)$). The same lemma also implies that $\#\big(\mathrm{Orb}(z) \cap \supp(\mu_f)\big)$ is uniformly bounded independently of $z \in \supp(\mu_f)$.
	
	Finally, since
	\[
	\mathrm{Orb}(z) = \Gamma \cap \pi^{-1}(z),
	\]
	where $\Gamma = \{ (\beta(w), w) \in N_1 \times \C^k \; ; \; \beta \in N_1,\; w \in \C^k \}$ and $\pi$ denotes the projection onto the second coordinate, it follows that the degree of $\mathrm{Orb}(z)$ is uniformly bounded in $z$.
\end{proof}
\begin{remark} \normalfont Picking $z \in \supp(\mu_f)$ to be a fixed point of $f$ (which exists after possibly replacing $f$ by an iterate), we obtain
	\[
	f(\mathrm{Orb}(z)) = \mathrm{Orb}(z).
	\]
	In dimension $2$, this already yields a contradiction, since Hénon maps do not preserve any (affine) algebraic set of dimension $1$ (see \cite{BedfordSmillie1}). However, this property does not extend to higher dimensions, as illustrated for instance in Example~\ref{example_notmeasure}.
\end{remark}

Let $\ell$ denote the dimension of $\mathrm{Orb}(z)$ for $z \in \supp(\mu_f)$. Let $\mathcal{C}$ be the variety of algebraic cycles in $\P^k$ of dimension $\ell$ and degree at most $M$ (see, e.g. \cite[Chapter~I]{Kollar}). We define $\mathcal{Z}$ to be the Zariski closure in $\mathcal{C}$ of the set
\[
\{\overline{\mathrm{Orb}(z)} \,;\, z \in \supp(\mu_f)\}.
\]
We then obtain the following lemma, whose proof relies on the Zariski density of $\supp(\mu_f)$ and is left to the reader.
\begin{lemma}\label{lem_Z}
\begin{itemize}
	\item For any $Z\in \mathcal{Z}$ and any $\beta \in N_1$, we have $\beta Z = Z$.
	\item For a generic $z \in \C^k$, there is a unique $Z\in \mathcal{Z}$ such that $z\in Z$.  
	\item the map $\mathrm{Orb}(z)\mapsto \mathrm{Orb}(f(z))$ extends to a birational map $\mathcal{F} : \mathcal{Z} \dasharrow \mathcal{Z}$. In particular, $\dim \mathcal{Z}+ \dim N_1 = k$.
\end{itemize}
\end{lemma}
Consequently, if $\Pi: \C^k \dasharrow \mathcal{Z}$ denotes the projection $z\mapsto Z$ where $z\in Z$, then $\Pi$ defines a meromorphic fibration and we have the commuting diagram
\begin{equation}\label{eq_diag}
\xymatrix{
	\P^k\ar@{-->}[rr]^{f}\ar@{-->}[d]^{\Pi} & & \ar@{-->}[d]^{\Pi}\P^k\\
 	\mathcal{Z}\ar@{-->}[rr]^{\mathcal{F}} & & 	\mathcal{Z} }.
\end{equation}

\begin{lemma}\label{lemma_entropylower}
	The topological entropy of $\mathcal{F}$ satisfies $h_{\mathrm{top}}(\mathcal{F}) <  \log d^s $.
\end{lemma}
\begin{proof} Let $d_p(\mathcal{F})$ denote the $p$-th dynamical degree of $\mathcal{F}$. 
	By \cite[Theorem 4]{Dang_degree} (see also \cite{Dinh_Vietanh, Dinh_VietAnh_Truong}), the first dynamical degree of $\mathcal{F}$ satisfies $d_1(\mathcal{F}) \le d$, and similarly $d_1(\mathcal{F}^{-1}) \le \delta$. Using the log-convexity of dynamical degrees (\cite{Khovanskii, Teissier}), we then obtain
	\[
	\max_{0 \le i \le \dim \mathcal{Z}} d_i(\mathcal{F}) < d^s,
	\]
	since $\dim \mathcal{Z} < k$. Gromov's entropy argument therefore implies
	\[
	h_{\mathrm{top}}(\mathcal{F})
	\le
	\log \Bigl(\max_{0 \le i \le \dim \mathcal{Z}} d_i(\mathcal{F})\Bigr)
	<
	\log d^s
	\]
	(see \cite{DS_entropy, Gromov_entropy}).  
\end{proof}

\begin{lemma}\label{lemma_entropyequal}
	The topological entropy of $\mathcal{F}$ satisfies $h_{\mathrm{top}}(\mathcal{F}) \geq  \log d^s $.
\end{lemma}
\begin{proof}Recall that $h_{\mathrm{top}}(f_{|\supp(\mu_f)}) = \log d^s$ by \cite[Theorem 2.3.8]{Sibony}. Let $\mathcal{Z}_\mu$ denote the collection of $Z \in \mathcal{Z}$ passing through points of $\supp(\mu_f)$. Restricting diagram \eqref{eq_diag} to $\supp(\mu_f)$, we obtain
	\begin{equation*}
		\xymatrix{
			\supp(\mu_f)\ar[rr]^{f}\ar[d]^{\Pi} & & \ar[d]^{\Pi} \supp(\mu_f)\\
			\mathcal{Z}_\mu\ar[rr]^{\mathcal{F}} & & \mathcal{Z}_\mu }.
	\end{equation*}
	Observe that the maps $f$, $\Pi$, and $\mathcal{F}$ are all well-defined and continuous by Lemma~\ref{lem_Z}. Since $\Pi$ is finite-to-one on $\supp(\mu_f)$, Bowen's entropy formula for fibrations \cite[Theorem 17]{Bowen} yields
	\[
	h_{\mathrm{top}}(\mathcal{F}_{|\mathcal{Z}_\mu})
	\ge
	h_{\mathrm{top}}(f_{|\supp(\mu_f)})
	=
	\log d^s.
	\]
\end{proof}
\begin{proof}[Proof of Proposition~\ref{prop_invariance}] Recall that we assumed $\dim N_1 > 0$. However, Lemmas~\ref{lemma_entropylower} and \ref{lemma_entropyequal} together yield a contradiction. Therefore, we must have $\dim N_1 = 0$, and hence $N_1$ is finite.
\end{proof}

\begin{corollary}\label{cor_finite} The set $N$ is finite. 
\end{corollary}
\begin{proof} 
	Pick \(\beta \in N\) and \(n\in\mathbb{N}\). Since \(N\) is a group and by the chain rule, we have \(f^{n}\circ \beta \circ f^{-n}\circ \beta^{-1}\in N_{1}\). By Proposition~\ref{prop_invariance}, \(N_{1}\) is finite, so there exist \(n\neq m\) such that \(f^{n}\circ \beta \circ f^{-n}\circ\beta^{-1}=f^{m}\circ \beta \circ f^{-m}\circ\beta^{-1}\). We deduce	
	\[
	\beta^{-1}\circ f^{\,n-m}\circ \beta = f^{\,n-m}.
	\]	
 In particular, since $ G_{f^{n-m}}^{+}= G_f^{+}$ and \( G^{+}_{\beta^{-1}\circ f^{n-m} \circ \beta}= G_{f^{n-m}}^{+}\circ \beta \), we deduce $G_f^{+}\circ \beta = G_f^{+}$ and similarly $G_f^{-}\circ \beta = G_f^{-}$. Then Lemma~\ref{lemma_affine} implies that there are only finitely many such \(\beta\).	
	\end{proof}

\begin{proof}[Proof of Theorem~\ref{tm_main}] Theorem~\ref{tm_main} is now a direct consequence of Lemma~\ref{lemma_group} and Corollary~\ref{cor_finite}.
\end{proof}

\begin{remark} \normalfont
	The finiteness of $N$ can be formulated over an arbitrary base field, not necessarily $\C$, for instance in positive characteristic. However, our approach relies heavily on complex-analytic methods. It would be interesting to obtain a purely algebraic proof of this result.
\end{remark}

\begin{remark} \normalfont
	In the statement of Theorem~\ref{tm_main}, we assume that $f$ and $g$ are Hénon--Sibony maps. It would be natural to seek an extension to the case where $f$ and $g$ become Hénon--Sibony maps only after passing to a suitable birational model. However, it is not clear in general why two maps sharing the same forward Julia set should be simultaneously conjugate to Hénon--Sibony maps in a common birational model (see \cite{DF_Manin} for related results in dimension $2$).
\end{remark}

\section{Rigidity of $\{G_f^+=a\}$ }\label{sec5}
	The proof of Theorem~\ref{tm_main} would be simpler if one knew that $\{G_f^+=a\}$ supports a unique positive closed $(s,s)$-current of mass $1$. Indeed, Lemma~\ref{lemma_constant} would then almost directly imply that $G_f^+ \circ \beta = 0$ on the support of $T^+_{f,s}$, so that $\beta$ preserves both $G_f^+$ and $G_f^-$. We will see that this property holds when $s=1$ (hence in dimension $2$), but that counterexamples exist in higher dimensions.
	
	When $a=0$, the following result is due to Forn\ae ss and Sibony~\cite{FS2} in the case $k=2$, and to Dinh and Sibony~\cite{superpotentiels} in general (without any assumption on $s$). For $a>0$ and $k=2$, Dinh and Sibony~\cite[Proposition~6.11]{DinhSibony_rigidity} proved the result under the additional assumption that $K_f^+$ is not contained in a (possibly singular) real-analytic hypersurface. As mentioned in the introduction, the general proof was communicated to us by Dinh~\cite{Dinh_personal}.
	
	\begin{proposition}\label{prop_Dinh}
		Let $f$ be a Hénon--Sibony map of $\C^k$ of degree $d$ with $s=1$. Then, for any $a \geq 0$, the current $T_a = \dd^c G_a$ is the unique positive closed current of mass $1$ supported on $\{G_f^+ = a\}$.
	\end{proposition}
	
	\begin{proof}Pick $a>0$. Let $(T_{1,n})$ and  $(T_{2,n})$ be two sequences of positive closed currents of mass $1$ supported on  $\{ G^+_f= a d^n \}$. We want to show that 
	\[\lim_{n\to \infty} d^{-n}(f^n)^*( T_{1,n}-T_{2,n})=0.\]
	Since $\dim I^-_f = a-1=0$, $I^-_f$ is reduced to a point. We pick homogeneous coordinates $[z_1:z_2:\dots:z_{k+1}]$ on $\P^k$ so that the hyperplane at infinity is given by $z_{k+1}=0$ with $I^-_f=[1:0:\dots:0]$, and we pick the corresponding coordinates $(z_1,\dots,z_k)$ on $\C^k$.  
	Pick $z$ such that $G_f^+(z)=c>a$, then $G_f^+(f^n(z))=d^n c$. Let $L$ be the complex line joining $f^n(z)$ and $I^-_f$. We take $z_1$ as the coordinate on $L$ and we write 
	$f^n(z)=(z_1^n, \dots, z_k^n)$ so $f^n(z)$ has coordinate $z_1^n$ on $L$.

	From the fact that $G_f^+(w)=\log^+\|w\|+ O(1)$ in a neighborhood of $I^-_f$ and   $G_f^-(w)\leq \log^+\|w\|+ O(1)$, one deduces from $G_f^+(f(w))= d G_f^+(w)$ and $G_f^-(f(w))= \delta^{-1} G_f^-(w)$ that there exist a neighborhood $V$ of $I^-_f$ in $\P^k$  such that $\log|z_1|= \log^+\|w\|+O(1)$ and $\log^+\|w\|+O(1)= G^+_f > G_f^-$ on $V\cap \C^k$.

	Hence, provided $n$ is large enough, $f^n(z)\in V$ and 
	$G_f^+(f^n(z))= \log|z_1^n| + O(1)$  while for any $z'\in  V \cap \C^k$, such that $G_f^+(z')=a d^n$,  one has $ \log|z_1'| \leq G_f^+(z') +O(1) =a d^n + O(1)$. On the other hand, we have that there exist a constant $C>1$ such that for any $z' \in L \backslash V$,
	$C^{-1}\leq | 1- \frac{z_1'}{z_1^n}|\leq C $ (simply because $z^n_1$ is very large and that outside $V$, $z'_1$ can not be too large). Combining those two facts gives, changing $C$ if necessary,  
	\begin{equation}\label{eq_comparaison}
		\forall z' \in L \cap \{ G^+_f= ad^n \}, \	   C^{-1}\leq | 1- \frac{z_1'}{z_1^n}|\leq C. 
	\end{equation} 
	Since $T_{1,n}-T_{2,n}$ is smooth near $I^-_f$, we can write $ T_{1,n}-T_{2,n}= \dd^c u_n$ with $u_n$ uniquely determined by $u_n(I^-_f)=0$. We thus have by Stokes and the fundamental solution of the Laplacian on $L$
	\[u_n(f^n(z)) = \langle u_n, [L] \wedge \dd^c \log | z'_1-z_1^n|  \rangle = \langle  (T_{1,n}-T_{2,n})\wedge [L] ,  \log | z'_1-z_1^n|  \rangle. \]
	Using that $(T_{i,n}\wedge [L]$ are probability measures, we deduce:
	\[u_n(f^n(z)) = \left\langle  (T_{1,n}-T_{2,n})\wedge [L] ,  \log \left| 1-\frac{z'_1}{z^n_1} \right|  \right\rangle. \]
	By \eqref{eq_comparaison}, we deduce $ C^{-1} \leq u_n(f^n(z)) \leq C$ so letting $n\to \infty$, we have 
	\[\lim_{n\to \infty} d^{-n} u_n(f^n(z))=0.\]
	Pick $T_{1,n}= T_{a d^n}$ and let $L'$ be a generic line passing trough $I^-_f$. Then, $d^{-n}(f^n)^* T_{a d^n}= T_{a}$ and  $T_{a} \wedge[L']$ is the Green measure $\mu_{a,L'}$ of the compact set $\{ G^+_f= a \}\cap L'$ in $L'$. In particular, when extracting a converging subsequence $T_{2,n}\to S$, we have $\mu_{n,2} :=T_{2,n}\wedge [L']\to S \wedge [L']=: \mu_{S,L'}$ ($L'$ is generic) hence $\mu_{S,L'}$ is a probability measure supported on $L' \cap\{ G^+_f= a \} $. By the above, we have, on $L'$, $ \mu_{a,L'}-\mu_{2,n}= \dd^c d^{-n}u_n(f^n(z))$  and we can extract a converging subsequence in $L^1_{\mathrm{loc}}$ so if we write $\mu_{a,L'}- \mu_{S,L'}=\dd^c v$, we have that $v \equiv 0$ in $\{G^+_f >a \} \cap L'$ (up to subtracting a constant). By extremality of the Green measure, this implies that $\mu_{a,L'}=\mu_{S,L'}$  so that  $v \equiv 0$ in $\{G^+_f >a \}\cap L'$. This implies that $T_{a}=T_S$ thus $\lim_{n\to \infty} d^{-n}(f^n)^*( T_{1,n}-T_{2,n})=0$ for any sequences $(T_{1,n})$ and  $(T_{2,n})$ of positive closed currents of mass $1$ supported on  $\{ G^+_f= a d^n \}$.

	The proof of the proposition is now classical:  pick  $S$ a positive closed current of mass $1$ supported by $\{ G^+_f= a \}$, then $S_n:= d^n (f^n)_*(S)$ is a positive closed current of mass $1$ supported on $\{ G^+_f= a d^n \}$ so by the above $S=d^{-n}(f^n)^*S_n \to T_a$.     
\end{proof}

\begin{example}\normalfont 	
Let $f$ and $g$ be two Hénon maps of $\C^2$ of degree $d$, and set
\[
F(x,y,z,w)= \bigl(f(x,y),\, g(z,w)\bigr).
\]
Then $F$ is a Hénon--Sibony map of $\C^4$, with $\dim I^+_F = \dim I^-_F = 1$, so $s=2$ and
\[
G^+_F = \max(G^+_f, G^+_g).
\]
Moreover,
\[
(\dd^c \max(G^+_F,a))^2
= \dd^c \max(G^+_f,a)\wedge \dd^c \max(G^+_g,a).
\]
For $b<a$, the current
\[
\dd^c \max(G^+_f,b)\wedge \dd^c \max(G^+_g,a)
\]
is supported on $\{G^+_f=b\}\times\{G^+_g=a\} \subset \{G^+_F=a\}$, but differs from $(\dd^c \max(G^+_F,a))^2$. Hence $\{G^+_F=a\}$ carries infinitely many distinct positive closed currents of mass $1$.
\end{example}

\section{An example}\label{sec6}
Consider the Hénon--Sibony map of $\C^3$ given by 
\[
F(x,y,z) = (y + x^2,\, z + y^2,\, x),
\]
so that
\[
F^{-1}(x,y,z) = \bigl(z,\, x - z^2,\, y - (x - z^2)^2\bigr).
\]
We have $I^+_F = \{[0:0:1:0]\}$ and $I^-_F = \{[x:y:0:0] \mid (x,y)\neq (0,0)\}$, hence $I^+_F \cap I^-_F = \varnothing$. We aim to determine explicitly
\[
N = \left\{ \beta \in \mathrm{Aff}_3 \;\middle|\; \forall n \in \N,\; F^n \circ \beta \circ F^{-n} \in \mathrm{Aff}_3 \right\}.
\]
As observed above, any $\beta \in N$ fixes both $I^+_F$ and $I^-_F$, and therefore can be written in the form
\[
\begin{pmatrix}
	a & b & 0 \\
	c & d & 0 \\
	0 & 0 & e
\end{pmatrix}
\begin{pmatrix}
	x \\ y \\ z
\end{pmatrix}
+
\begin{pmatrix}
	x_0 \\ y_0 \\ z_0
\end{pmatrix}.
\]
A direct computation yields
\[
(F \circ \beta \circ F^{-1})(x,y,z)
=
\begin{pmatrix}
	c z + d(x - z^2) + y_0 + (a z + b(x - z^2) + x_0)^2 \\
	e\!\left(y - (x - z^2)^2\right) + z_0 + (c z + d(x - z^2) + y_0)^2 \\
	a z + b(x - z^2) + x_0
\end{pmatrix}.
\]
Since $F \circ \beta \circ F^{-1} \in N$, inspection of the third coordinate gives $b=0$ and $a\neq 0$. Hence
\[
(F \circ \beta \circ F^{-1})(x,y,z)
=
\begin{pmatrix}
	c z + d(x - z^2) + y_0 + (a z + x_0)^2 \\
	e\!\left(y - (x - z^2)^2\right) + z_0 + (c z + d(x - z^2) + y_0)^2 \\
	a z + x_0
\end{pmatrix}.
\]
Examining the first coordinate yields $c = -2a x_0$ and $d = a^2$, while the second coordinate gives $e = d^2 = a^4$. Thus
\[
(F \circ \beta \circ F^{-1})(x,y,z)
=
\begin{pmatrix}
	a^2 x + y_0 + x_0^2 \\[6pt]
	a^4 y - 4 a^3 x_0 x z + 4 a^3 x_0 z^3 
	+ 4 a^2 x_0^2 z^2 + 2 a^2 y_0 x - 2 a^2 y_0 z^2 \\
	\qquad - 4 a x_0 y_0 z + y_0^2 + z_0 \\[6pt]
	a z + x_0
\end{pmatrix}.
\]
Since $a \neq 0$, the presence of the nonlinear term $-4 a^3 x_0 x z$ in the second coordinate forces $x_0 = 0$. It then follows (for instance from the first and third coordinates) that $y_0 = 0$ and $z_0 = 0$. Hence
\[
(F \circ \beta \circ F^{-1})(x,y,z)
=
\begin{pmatrix}
	a^2 x \\
	a^4 y \\
	a z
\end{pmatrix}.
\]
Applying the same argument to $F \circ \beta \circ F^{-1}$ shows that $a = (a^2)^4$, hence $a^7 = 1$. In conclusion,
\[
N = \left\{
\begin{pmatrix}
	a & 0 & 0 \\
	0 & a^2 & 0 \\
	0 & 0 & a^4
\end{pmatrix}
\begin{pmatrix}
	x \\ y \\ z
\end{pmatrix}
\;\middle|\; a^7 = 1
\right\}.
\]
One can further observe that for $a \neq 1$, $\beta $ and $F$ do not commute as we have
\[
\beta \circ F \circ \beta^{-1}(x,y,z)
=
\bigl(a^{-2}(x^2 + y),\, a^3(z + y^2),\, a^{-1}x\bigr)
\neq F(x,y,z).
\]

  \subsection*{Acknowledgment} 
  The author is grateful to T.C. Dinh for communicating the proof of Proposition~\ref{prop_Dinh}. He also thanks I. Marin, D. Chataur and L. Hennecart for helpful discussions.

\end{document}